\documentclass[12pt,a4paper]{article}

\usepackage{graphicx,tikz}
\usepackage{latexsym}
\usepackage{amsmath}
\usepackage{amssymb}
\usepackage{arydshln}

\usepackage{cite}
\usepackage{stmaryrd}
\usepackage{enumerate}

\usepackage{hyperref}

\usepackage{color}
\usepackage{lineno}
\usepackage{graphicx}
\usepackage{ae}
\usepackage{amsmath}
\usepackage{amssymb}
\usepackage{latexsym}
\usepackage{url}
\usepackage{epsfig}
\usepackage{mathrsfs}
\usepackage{amsfonts}
\usepackage{amsthm}
\usepackage{bm}
\newcommand{\Cay}{\mathrm{Cay}}

\newcommand{\GL}{\mathrm{GL}}

\newtheorem{theorem}{Theorem}[section]

\newtheorem{lemma}[theorem]{Lemma}

\newtheorem{example}{Example}

\theoremstyle{definition}

\numberwithin{equation}{section} %% ��ʽ���Ÿ����½�
\allowdisplaybreaks    %% ��ʽ��ҳչʾ

\def\qed{\hfill$\Box$\vspace{12pt}}

\long\def\delete#1{}

\usepackage{xcolor}
\usepackage[normalem]{ulem}

\usepackage{stfloats}
\begin{document}
\title{Pretty good fractional revival on Cayley graphs over dicyclic groups}
\author{Jing Wang$^{a,b,c}$,~Ligong Wang$^{a,b,c}$$^,$\thanks{Supported by the National Natural Science Foundation of China (Nos. 11871398 and 12271439).}~,~Xiaogang Liu$^{a,b,c}$\thanks{Supported by the National Natural Science Foundation of China (No. 12371358) and the Guangdong Basic and Applied Basic Research Foundation (No. 2023A1515010986).}~$^,$\thanks{Corresponding author. Email addresses: wj66@mail.nwpu.edu.cn, lgwangmath@163.com, xiaogliu@nwpu.edu.cn}
\\[2mm]
{\small $^a$School of Mathematics and Statistics,}\\[-0.8ex]
{\small Northwestern Polytechnical University, Xi'an, Shaanxi 710072, P.R.~China}\\
{\small $^b$Research \& Development Institute of Northwestern Polytechnical University in Shenzhen,}\\[-0.8ex]
{\small Shenzhen, Guandong 518063, P.R. China}\\
{\small $^c$Xi'an-Budapest Joint Research Center for Combinatorics,}\\[-0.8ex]
{\small Northwestern Polytechnical University, Xi'an, Shaanxi 710129, P.R. China}\\
}
\date{}

\date{}

\openup 0.5\jot
\maketitle

\begin{abstract}
%In this paper, we investigate the existence of pretty good fractional revival on Cayley graphs over dicyclic groups. We find that if $n$ is a power of a prime number, then $\Cay(T_{4n}, S)$ admits pretty good fractional revival for some subset $S$ in $T_{4n}$. Moreover, we show that $\Cay(T_{4n}, S)$ doesn't admit pretty good fractional revival in other cases.
In this paper, we investigate the existence of pretty good fractional revival on Cayley graphs over dicyclic groups. We first give a necessary and sufficient description for Cayley graphs over dicyclic groups admitting pretty good fractional revival. By this description, we give some sufficient conditions for Cayley graphs over dicyclic groups admitting or not admitting pretty good fractional revival.
\smallskip


\end{abstract}

\section{Introduction}
Let $\Gamma$ be a simple and undirected graph with vertex set $V(\Gamma)$ and edge set $E(\Gamma)$. The \emph{adjacency matrix} of $\Gamma$ is denoted by $A=A(\Gamma)=(a_{u,v})_{u,v\in V(\Gamma)}$, where $a_{u,v}=1$ if $u$ and $v$ are adjacent, and $a_{u,v}=0$ otherwise. The \emph{transition matrix} \cite{FarhiG98} of $\Gamma$ with respect to $A$ is defined by
$$
H(t) = \exp(\imath tA)=\sum_{k=0}^{\infty}\frac{\imath^{k} A^{k} t^{k}}{k!}, ~ t \in \mathbb{R},~\imath=\sqrt{-1},
$$
where $\mathbb{R}$ is the set of real numbers.
Let $\mathbb{C}^{|\Gamma|}$ denote the $|\Gamma|$-dimensional vector space over complex field and $\mathbf{e}_u$ the standard basis vector in $\mathbb{C}^{|\Gamma|}$ indexed by the vertex $u$.
If $u$ and $v$ are distinct vertices in $\Gamma$ and there is a time $t$ such that
\begin{equation*}\label{FRCAY-EQUATION3}
H(t)\mathbf{e}_u=\alpha\mathbf{e}_u+\beta\mathbf{e}_v,
\end{equation*}
where $\alpha$ and $\beta$ are complex numbers and $|\alpha|^2+|\beta|^2=1$, then we say that $\Gamma$ admits \emph{fractional revival} (FR for short) from $u$ to $v$ at time $t$. In particular, if $\alpha=0$, we say $\Gamma$ admits \emph{perfect state transfer} (PST for short) from $u$ to $v$ at time $t$; and if $\beta=0$, we say $G$ is \emph{periodic}  relative to the vertex $u$ at time $t$.

Quantum state transfer was first introduced by Bose in \cite{Bose03}. It is a very important research content for quantum communication protocols. Up until now, many wonderful results on FR (or PST) have been achieved, including trees \cite{CoutinhoL2015, GodsilKSS12,GenestVZ16,CCTVZ19}, Cayley graphs \cite{Basic09, CaoCL20, CaoF21, CC, LiLZZ21,  Tan19, Tm19,WangWliu22,WangWliu23}, distance regular graphs \cite{Coutinho15,ChanCCTVZ20} and some graph operations \cite{chris1, chris2, LiuW2021}. For more information, we refer the reader to \cite{Coh14,Coh19, CGodsil, Godsil12,GodsilZ22}.

It is known that graphs with a given maximum valency with FR (or PST) are rare \cite{Godsil12,CCTVZ19}. Thus, in \cite{CGodsil}, Godsil proposed to consider a relaxation of PST, named as pretty good state transfer. A graph $G$ is said to admit \emph{pretty good state transfer} (PGST for short) from vertex $u$ to vertex $v$ if there is a sequence $\{t_k\}$ of real numbers such that
$$\lim\limits_{k\rightarrow\infty}H(t_k)\mathbf{e}_u=\gamma\mathbf{e}_v.$$
where $\gamma$ is a complex number and $|\gamma|=1$.
PGST has been studied on many families of graphs, such as trees \cite{Ban}, double stars \cite{Fan}, circulant graphs \cite{HPal3, HPal5}, Cayley graphs over dihedral groups \cite{CaoWF20} and Cayley graphs over semi-dihedral groups \cite{WangC19}.

Recently, Chan et al. \cite{ChanDEKL21} relaxed the condition for FR to give the definition of pretty good fractional revival. If there is a sequence $\{t_k\}$ of real numbers such that
$$\lim\limits_{k\rightarrow\infty}H(t_k)\mathbf{e}_u=\alpha\mathbf{e}_u+\beta\mathbf{e}_v,$$
where $\alpha$ and $\beta$ are complex numbers and $|\alpha|^2+|\beta|^2=1$, then we say that $\Gamma$ admits \emph{pretty good fractional revival} (PGFR for short) from $u$ to $v$ with respect to the time sequence $\{t_k\}$. In particular, if $\beta=0$, we say $G$ is \emph{approximate periodic} at vertex $u$. Note that approximate periodic occurs for every vertex synchronously. So it is not interesting to consider the case when two vertices are approximately periodic.

Up until now, there are only few results on PGFR. In 2021, Chan et al. \cite{ChanDEKL21} gave complete characterizations of when paths and cycles admit PGFR. In 2022, Chan et al.  \cite{ChanJLSYZ22} replaced $A(\Gamma)$ in the definition of PGFR by Laplacian matrix $L(\Gamma)=D(\Gamma)-A(\Gamma)$, where $D(\Gamma)$ denotes  the diagonal matrix of vertex degrees of $\Gamma$, and they classified the paths and the double stars that have Laplacian PGFR.

In this paper, we consider the existence of PGFR on Cayley graphs over  non-abelian groups. More specifically, we consider PGFR on Cayley graphs over dicyclic groups. Let $G$ be a group and $S$ a subset of $G$ with $1_G\notin S=S^{-1}=\left\{s^{-1}\mid s\in S\right\}$ (inverse-closed). The \emph{Cayley graph} $\Gamma=\Cay(G,S)$ is a graph whose vertex set is $G$ and edge set is $\{\{g,sg\}\mid g\in G, s\in S\}$.
The \emph{dicyclic group} $T_{4n}$, where $n\geq2$, is given by
\begin{align*}
T_{4n}&=\langle a,b\mid a^{2n}=1, a^n=b^2, b^{-1}ab=a^{-1}\rangle\\
      &=\{a^k, a^kb\mid 0\leq k\leq 2n-1\}.
\end{align*}

The paper is organized as follows. In Section 2, we introduce some known results on the representation of dicyclic groups and spectral decomposition of Cayley graphs over finite groups.  In Section 3, we first give a necessary and sufficient description for Cayley graphs over dicyclic groups admitting PGFR. By this description, we give some sufficient conditions for Cayley graphs over dicyclic groups admitting or not admitting PGFR.

\section{Preliminaries}
In this section, we give some notions, notations and helpful results used in this paper.
\subsection{The representation of dicyclic groups}
Let $G$ be a finite group and $V$ a non-zero vector space over $\mathbb{C}$ with finite dimension. A \emph{representation} of $G$ on $V$ is a group homomorphism $\rho:G\rightarrow \GL(V)$, where $\GL(V)$ is the general linear group of $V$ defined as the group of invertible linear transformations of $V$. The \emph{degree} of $\rho$ is defined as the dimension of $V$. Two representations $\rho_i:G\rightarrow \GL(V_i)$, $i=1,2$ are said to be \emph{equivalent}, written $\rho_1\sim\rho_2$, if there exists a vector space isomorphism $T:V_1\rightarrow V_2$ such that, for all $g\in G$, $\rho_1(g)T=T\rho_2(g)$. A representation $\rho$ of $G$ is \emph{reducible} if there exists two representations $\rho_1, \rho_2$ of $G$ such that $\rho\sim \rho_1\oplus\rho_2$, and \emph{irreducible} otherwise.
As usual, by identifying each element of $\GL(V)$ with its matrix with respect to a chosen basis for $V$, we may identify $\GL(V)$ with the group $\GL(n,\mathbb{C})$ of all invertible $n\times n$ matrices over complex field $\mathbb{C}$ with operation the product of matrices, and we may view a representation of $G$ on $V$ as a group homomorphism $\rho:G\rightarrow \GL(n,\mathbb{C})$. The \emph{character} of $\rho$ is the mapping $\chi_{\rho}:G\rightarrow \mathbb{C}$ defined by
$$\chi_{\rho}(g)=\mathrm{Tr}(\rho(g)), g\in G,$$
where $\mathrm{Tr}$ denotes the trace of a matrix.

Let $G$ be a finite non-abelian group and
$\mathbb{C}^{G}=\{f\mid f:G\rightarrow\mathbb{C}\}$ be the group algebra of $G$.  Suppose that \{$\rho^{(1)}, \rho^{(2)},\ldots,\rho^{(s)}$\} is a complete set of unitary representatives of the equivalence classes of irreducible representations of $G$ and $d_h$ denotes the \emph{degree} of $\rho^{(h)}$. Define $T:\mathbb{C}^{G}\rightarrow\GL(d_1,\mathbb{C})\times\cdots\times \GL(d_s,\mathbb{C})$
 by $Tf=\left(\hat{f}(\rho^{(1)}),\ldots,\hat{f}(\rho^{(s)})\right)$ where
$$\hat{f}(\rho^{(h)})_{ij}=\sum\limits_{g\in G}f(g)\overline{\rho_{ij}^{(h)}(g)},$$
and $\overline{\rho_{ij}^{(h)}(g)}$ denotes the complex conjugate of $\rho_{ij}^{(h)}(g)$. $Tf$ is called the \emph{Fourier transform} of $f$. The \emph{Fourier inversion} \cite{Steinberg12} is
$$f=\frac{1}{|G|}\sum\limits_{i,j,h}d_h\hat{f}(\rho^{(h)})_{ij}\rho^{(h)}_{ij}.$$

The representations and characters of the dicyclic group $T_{4n}$ are given as follows \cite{JamesL01}.

\begin{lemma}\emph{(see \cite{JamesL01}, Exercises 17.6 and 18.3)}\label{pgfrdicy-lemma0}
Let $n\geq2$ and $\omega=\exp(\frac{\pi \imath}{n})$ be a $2n$-th root of unity. The irreducible representations and the character table of the dicyclic group $T_{4n}$ are listed in Tables $1,2,3$ and $4$.
\end{lemma}

\begin{table}[t]
\begin{center}
\caption{Irreducible representation of $T_{4 n}$ for $n$ even.}
\begin{tabular}{lll}
\hline                    ~~~~& $a$    &~~~~ $b$ \\
\hline$\psi_1$            ~~~~& $(1)$  &~~~~ $(1)$ \\
      $\psi_2$            ~~~~& $(1)$  &~~~~ $(-1)$ \\
      $\psi_3$            ~~~~& $(-1)$ &~~~~ $(1)$ \\
      $\psi_4$            ~~~~& $(-1)$ &~~~~ $(-1)$ \\
      $\rho_h,(1 \leq h \leq n-1)$ ~~& $\left(\begin{array}{cc}
                                    \omega^h & 0 \\
                                     0       & \omega^{-h}
                                     \end{array}\right)$
                                   ~~&~~ $\left(\begin{array}{cc}
                                        0         & 1 \\
                                     \omega^{h n} & 0\end{array}\right)$ \\
\hline
\end{tabular}
\end{center}
\end{table}

\begin{table}[t]
\begin{center}
\caption{Irreducible representation of $T_{4 n}$ for $n$ odd.}
\begin{tabular}{lll}
\hline                    ~~~~& $a$    &~~~~ $b$ \\
\hline$\psi_1$            ~~~~& $(1)$  &~~~~ $(1)$ \\
      $\psi_2$            ~~~~& $(1)$  &~~~~ $(-1)$ \\
      $\psi_3$            ~~~~& $(-1)$ &~~~~ $(\imath)$ \\
      $\psi_4$            ~~~~& $(-1)$ &~~~~ $(-\imath)$ \\
      $\rho_h,(1 \leq h \leq n-1)$ ~~& $\left(\begin{array}{cc}
                                    \omega^h & 0 \\
                                     0       & \omega^{-h}
                                     \end{array}\right)$
                                   ~~&~~ $\left(\begin{array}{cc}
                                        0         & 1 \\
                                     \omega^{h n} & 0\end{array}\right)$ \\
\hline
\end{tabular}
\end{center}
\end{table}

\begin{table}[t]
\begin{center}
\caption{Character Table of $T_{4 n}$ for $n$ even.}
\begin{tabular}{llllll}
\hline                    ~~~~& $1$    &~~~~ $a^n$     &~~~~ $a^k(1\leq k\leq n-1)$      &~~~~ $b$  &~~~~ $ab$ \\
\hline$\chi_1$            ~~~~& $1$    &~~~~ $1$       &~~~~ $1$                         &~~~~ $1$  &~~~~ $1$  \\
      $\chi_2$            ~~~~& $1$    &~~~~ $1$       &~~~~ $1$                         &~~~~ $-1$ &~~~~ $-1$ \\
      $\chi_3$            ~~~~& $1$    &~~~~ $1$       &~~~~ $(-1)^k$                    &~~~~ $1$  &~~~~ $-1$ \\
      $\chi_4$            ~~~~& $1$    &~~~~ $1$       &~~~~ $(-1)^k$                    &~~~~ $-1$ &~~~~ $1$  \\
      $\vartheta_h,(1 \leq h \leq n-1)$
                          ~~~~& $2$    &~~~~ $2(-1)^h$ &~~~~ $\omega^{hk}+\omega^{-hk}$  &~~~~ $0$  &~~~~ $0$  \\
\hline
\end{tabular}
\end{center}
\end{table}

\begin{table}[t]
\begin{center}
\caption{Character Table of $T_{4 n}$ for $n$ odd.}
\begin{tabular}{llllll}
\hline                    ~~~~& $1$    &~~~~ $a^n$      &~~~~ $a^k(1\leq k\leq n-1)$      &~~~~ $b$       &~~~~ $ab$ \\
\hline$\chi_1$            ~~~~& $1$    &~~~~ $1$        &~~~~ $1$                         &~~~~ $1$       &~~~~ $1$  \\
      $\chi_2$            ~~~~& $1$    &~~~~ $1$        &~~~~ $1$                         &~~~~ $-1$      &~~~~ $-1$ \\
      $\chi_3$            ~~~~& $1$    &~~~~ $-1$       &~~~~ $(-1)^k$                    &~~~~ $\imath$  &~~~~ $-\imath$ \\
      $\chi_4$            ~~~~& $1$    &~~~~ $-1$       &~~~~ $(-1)^k$                    &~~~~ $-\imath$ &~~~~ $\imath$  \\
      $\vartheta_h,(1 \leq h \leq n-1)$
                          ~~~~& $2$    &~~~~ $2(-1)^h$  &~~~~ $\omega^{hk}+\omega^{-hk}$  &~~~~ $0$        &~~~~ $0$  \\
\hline
\end{tabular}
\end{center}
\end{table}
\subsection{Spectral decomposition of Cayley graphs over finite groups}
Let $\Gamma$ be a graph with adjacency matrix $A$.  The eigenvalues of $A$ are called the \emph{eigenvalues} of $\Gamma$. Suppose that $ \lambda_1\geq\lambda_2\geq\cdots\geq\lambda_{|G|}$ ($\lambda_i$ and $\lambda_j$ may be equal) are all eigenvalues of $\Gamma$  and $\xi_j$ is the eigenvector associated with $\lambda_{j}$, $j=1,2,\ldots,|G|$.  Let $\mathbf{x}^H$ denote the conjugate transpose of a column vector $\mathbf{x}$. Then, for each eigenvalue $\lambda_j$ of $\Gamma$, define
$$
E_{\lambda_j} =\mathbf{\xi}_j (\mathbf{\xi}_j)^H,
$$
which is usually called the \emph{eigenprojector} corresponding to  $\lambda_j$ of $G$. Note that $\sum_{j=1}^{|G|}E_{\lambda_j}=I_{|G|}$ (the identity matrix of order $|G|$). Then
\begin{equation}\label{spect1}
A=A\sum_{j=1}^{|G|}E_{\lambda_j} =\sum_{j=1}^{|G|}A\mathbf{\xi}_j (\mathbf{\xi}_j)^H  =\sum_{j=1}^{|G|}\lambda_j\mathbf{\xi}_j (\mathbf{\xi}_j)^H  =\sum_{j=1}^{|G|}\lambda_jE_{\lambda_j},
\end{equation}
which is called the \emph{spectral decomposition of $A$ with respect to the eigenvalues}  (see ``Spectral Theorem for Diagonalizable Matrices'' in \cite[Page 517]{MAALA}). Note that $E_{\lambda_j}^{2}=E_{\lambda_j}$ and $E_{\lambda_j}E_{\lambda_h}=\mathbf{0}$ for $j\neq h$, where $\mathbf{0}$ denotes the zero matrix. So, by (\ref{spect1}), we have
\begin{equation}\label{SpecDec2-1}
H(t)=\sum_{k\geq 0}\dfrac{\imath^{k}A^{k}t^{k}}{k!}=\sum_{k\geq 0}\dfrac{\imath^{k}\left(\sum\limits_{j=1}^{|G|}\lambda_{j}^{k}E_{\lambda_j}\right)t^{k}}{k!} =\sum_{j=1}^{|G|}\exp(\imath t\lambda_{j})E_{\lambda_j}.
\end{equation}

The eigenvalues and eigenvectors of the adjacency matrix of a Cayley graph over a finite group are given in the following lemma.

\begin{lemma}\emph{(see \cite{Steinberg12}, Exercise 5.12.3)}\label{pgfrdicy-lemma1}
Let $G=\{g_1,g_2,\ldots,g_{|G|}\}$ be a finite group of order $|G|$ and let $\{\rho^{(1)}, \rho^{(2)},\ldots, \rho^{(s)}\}$ be a complete set of unitary representatives of the equivalent classes of irreducible representations of $G$. Let $\chi_h$ be the character of $\rho^{(h)}$ and $d_h$ be the degree of $\rho^{(h)}$. Suppose that $S$ is a symmetric set with $gSg^{-1}=S$ for all $g\in G$. Then the eigenvalues of the adjacency matrix of a Cayley graph $\Cay(G,S)$ are given by
$$\lambda_h=\frac{1}{d_h}\sum\limits_{g\in S}\chi_h(g),~ 1\leq h\leq s.$$
Each $\lambda_h$ has multiplicity $d_h^2$. Moreover, the vectors
$$\xi_{ij}^{(h)}=\frac{\sqrt{d_h}}{\sqrt{|G|}}\left(\rho_{ij}^{(h)}(g_1),\rho_{ij}^{(h)}(g_2),\ldots,\rho_{ij}^{(h)}(g_{|G|})\right)^T, ~1\leq i,j\leq d_h,$$
form an orthonormal basis orthogonal basis for the eigenspace $V_{\lambda_h}$.
\end{lemma}

Let $G$ be a finite group, and $S$ a symmetric set with $gSg^{-1}=S$. Let $\Gamma=\Cay(G,S)$ be a Cayley group over $G$. By Lemma \ref{pgfrdicy-lemma1} and the definition of eigenprojector, we have
\begin{equation*}\label{pgfrdicy-equation1}
(E_{\lambda_h})_{u,v}=\frac{d_h}{|G|}\rho_{ij}^{(h)}(u)\overline{\rho_{ij}^{(h)}(v)}, ~1\leq i,j\leq d_h,
\end{equation*}
with respect to the eigenvalue $\lambda_h$ of $\Gamma$, $1\leq h\leq s$. Hence, by  (\ref{SpecDec2-1}), the spectral decomposition of $\Gamma=\Cay(G,S)$ is
\begin{equation}\label{pgfrdicy-equation*}
H(t)_{u,v}=\frac{1}{|G|}\sum\limits_{i,j,h}d_h\exp(\imath t \lambda_h)\rho_{ij}^{(h)}(u)\overline{\rho_{ij}^{(h)}(v)}.
\end{equation}

\section{PGFR on Cayley graphs over dicyclic groups}
Recall that if $\Gamma=\Cay(T_{4n}, S)$ admits PGFR from $u$ to $v$, then there is a sequence $\{t_k\}$ of real numbers such that
\begin{equation}\label{pgfrdicy-equation2}
\lim_{k\rightarrow\infty}H(t_k)\mathbf{e}_u=\alpha\mathbf{e}_u+\beta\mathbf{e}_v,
\end{equation}
where $\alpha$ and $\beta$ are complex numbers and $|\alpha|^2+|\beta|^2=1$. Write $\lim_{k\rightarrow\infty}H(t_k)=U$. Note that $U$ is a unitary matrix. Then by (\ref{pgfrdicy-equation2}), the $u$-th row of $U$ is determined by
\begin{equation*}
U_{u,w}=\left\{
\begin{array}{rl}
\alpha, &\text{if}~w=u,\\[0.2cm]
\beta,  &\text{if}~w=v,\\[0.2cm]
0,      &\text{otherwise}.
\end{array}\right.
\end{equation*}
On the other hand, by (\ref{pgfrdicy-equation*}),
\begin{equation*}
U_{u,w}=\lim_{k\rightarrow\infty}H(t_k)_{u,w}=\frac{1}{4n}\sum\limits_{i,j,h}d_h\lim_{k\rightarrow\infty}\exp(\imath t_k \lambda_h)\rho_{ij}^{(h)}(u)\overline{\rho_{ij}^{(h)}(w)}.
\end{equation*}
Write
$$\hat{f}(\rho^{(h)})_{ij}=\lim_{k\rightarrow\infty}\exp(\imath t_k \lambda_h)\rho_{ij}^{(h)}(u), ~ 1\leq i,j\leq d_h, ~1\leq h\leq s,$$
and $f=U_{u,w}$. Then by its Fourier transform, for $1\leq i,j\leq d_h$, $1\leq h\leq s$, we have
\begin{equation}\label{pgfrdicy-equation4}
\lim_{k\rightarrow\infty}\exp(\imath t_k \lambda_h)\rho_{ij}^{(h)}(u)=\sum\limits_{w\in T_{4n}}U_{u,w}\rho_{ij}^{(h)}(w)=\alpha\rho_{ij}^{(h)}(u)+\beta\rho_{ij}^{(h)}(v).
\end{equation}

For the sake of convenience, we label $0\leq u\leq2n-1$ as the point $a^u$, and $2n\leq u\leq4n-1$ as the point $a^ub$. Note that the representations of $\Cay(T_{4n}, S)$ are shown in Tables $1$ and $2$. We label $\lambda_i~(1\leq i\leq 4)$ as the eigenvalues corresponding to the representations $\psi_i ~(1\leq i\leq 4)$ and label $\mu_h ~(1\leq h\leq n-1)$ as the eigenvalues corresponding to the representations $\rho_h ~(1\leq h\leq n-1)$, respectively. Plugging them in (\ref{pgfrdicy-equation4}),
\begin{itemize}
\item[\rm(1)]for $0\leq u,v\leq2n-1$ or $2n\leq u,v\leq4n-1$, we have
\begin{equation}\label{pgfrdicy-equation5}
\left\{
\begin{aligned}
&\lim_{k\rightarrow\infty}\exp(\imath t_k\lambda_1)=\alpha+\beta,\\
&\lim_{k\rightarrow\infty}\exp(\imath t_k\lambda_2)=\alpha+\beta,\\
&\lim_{k\rightarrow\infty}\exp(\imath t_k\lambda_3)=\alpha+(-1)^{u+v}\beta,\\
&\lim_{k\rightarrow\infty}\exp(\imath t_k\lambda_4)=\alpha+(-1)^{u+v}\beta,\\
&\lim_{k\rightarrow\infty}\exp(\imath t_k\mu_h)=\alpha+\omega^{(v-u)h}\beta, ~1\leq h\leq n-1,\\
&\lim_{k\rightarrow\infty}\exp(\imath t_k\mu_h)=\alpha+\omega^{(u-v)h}\beta, ~1\leq h\leq n-1;\\
\end{aligned}
\right.
\end{equation}

\item[\rm(2)] for $0\leq u\leq2n-1, 2n\leq v\leq4n-1$, we have
\begin{equation*}
\left\{
\begin{aligned}
&\lim_{k\rightarrow\infty}\exp(\imath t_k\lambda_1)=\alpha+\beta,\\
&\lim_{k\rightarrow\infty}\exp(\imath t_k\lambda_2)=\alpha-\beta,\\
&\lim_{k\rightarrow\infty}\exp(\imath t_k\lambda_3)=\left\{
\begin{array}{ll}
  \alpha+(-1)^{u+v}\beta, &\text{if~} n \text{~is~even},\\[0.2cm]
\alpha+(-1)^{u+v}\beta\imath, &\text{if~} n \text{~is~odd},\\
\end{array}\right.\\
&\lim_{k\rightarrow\infty}\exp(\imath t_k\lambda_4)=\left\{
\begin{array}{ll}
\alpha+(-1)^{u+v+1}\beta, &\text{if~} n \text{~is~even},\\[0.2cm]
\alpha+(-1)^{u+v+1}\beta\imath, &\text{if~} n \text{~is~odd},\\
\end{array}\right.\\
&\lim_{k\rightarrow\infty}\exp(\imath t_k\mu_h)=\alpha, ~1\leq h\leq n-1,\\
&\omega^{vh}\beta=(-1)^h\omega^{-vh}\beta=0, ~1\leq h\leq n-1;\\
\end{aligned}
\right.
\end{equation*}

\item[\rm(3)] for $2n\leq u\leq4n-1, 0\leq v\leq2n-1$, we have
\begin{equation*}
\left\{
\begin{aligned}
&\lim_{k\rightarrow\infty}\exp(\imath t_k\lambda_1)=\alpha+\beta,\\
&\lim_{k\rightarrow\infty}\exp(\imath t_k\lambda_2)=\alpha-\beta,\\
&\lim_{k\rightarrow\infty}\exp(\imath t_k\lambda_3)=\left\{
\begin{array}{ll}
\alpha+(-1)^{u+v}\beta, &\text{if~} n \text{~is~even},\\[0.2cm]
\alpha-(-1)^{u+v}\beta\imath, &\text{if~} n \text{~is~odd},\\
\end{array}\right.\\
&\lim_{k\rightarrow\infty}\exp(\imath t_k\lambda_4)=\left\{
\begin{array}{ll}
\alpha+(-1)^{u+v+1}\beta, &\text{if~} n \text{~is~even},\\[0.2cm]
\alpha-(-1)^{u+v+1}\beta\imath, &\text{if~} n \text{~is~odd},\\
\end{array}\right.\\
&\lim_{k\rightarrow\infty}\exp(\imath t_k\mu_h)=\alpha,~1\leq h\leq n-1,\\
&\omega^{vh}\beta=\omega^{-vh}\beta=0,~1\leq h\leq n-1.\\
\end{aligned}
\right.
\end{equation*}
\end{itemize}

Since $\beta\neq0$, $\Cay(T_{4n}, S)$ does not admit PGFR between vertices $u$ and $v$ if $0\leq u\leq2n-1$ and $2n\leq v\leq4n-1$ or $2n\leq u\leq4n-1$ and $0\leq v\leq2n-1$. Thus we only need to consider the first case with $0\leq u,v\leq2n-1$ or $2n\leq u,v \leq4n-1$. By the last two equations of (\ref{pgfrdicy-equation5}), we have
$$\omega^{(u-v)h}=\omega^{(v-u)h}, ~\text{for}~1\leq h\leq n-1,$$
which implies $n\mid (u-v)$. That is, $u-v=n$. Thus, (\ref{pgfrdicy-equation5}) is equivalent to the following:
\begin{equation}\label{pgfrdicy-equation6}
\left\{
\begin{aligned}
&\lim_{k\rightarrow\infty}\exp(\imath t_k\lambda_1)=\alpha+\beta,\\
&\lim_{k\rightarrow\infty}\exp(\imath t_k\lambda_2)=\alpha+\beta,\\
&\lim_{k\rightarrow\infty}\exp(\imath t_k\lambda_3)=\alpha+(-1)^n\beta,\\
&\lim_{k\rightarrow\infty}\exp(\imath t_k\lambda_4)=\alpha+(-1)^n\beta,\\
&\lim_{k\rightarrow\infty}\exp(\imath t_k\mu_h)=\alpha+(-1)^h\beta, ~1\leq h\leq n-1.\\
\end{aligned}
\right.
\end{equation}

Suppose that $\delta_1, \delta_2$ are real numbers satisfying $\alpha+\beta=\exp(\imath \delta_1)$ and $\alpha-\beta=\exp(\imath \delta_2)$. We consider the following two cases.

\noindent\emph{Case 1.} $n$ is odd.
%\noindent\emph{Case 1.} $n$ is odd.
%\begin{equation}\label{pgfrdicy-equation6}
%\left\{
%\begin{aligned}
%&\lim_{k\rightarrow\infty}\exp(\imath (\lambda_2-\lambda_1)t_k)=1,\\
%&\lim_{k\rightarrow\infty}\exp(\imath (\lambda_3-\lambda_1)t_k)=\frac{\alpha-\beta}{\alpha+\beta},\\
%&\lim_{k\rightarrow\infty}\exp(\imath (\lambda_4-\lambda_1)t_k)=\frac{\alpha-\beta}{\alpha+\beta},\\
%&\lim_{k\rightarrow\infty}\exp(\imath (\mu_h-\lambda_1)t_k)=\frac{\alpha-\beta}{\alpha+\beta}, &\text{if}~h~\text{is~odd},\\
%&\lim_{k\rightarrow\infty}\exp(\imath (\mu_h-\lambda_1)t_k)=1, &\text{if}~h~\text{is~even}.\\
%\end{aligned}
%\right.
%\end{equation}
%
%\noindent\emph{Case 2.} $n$ is even.
%\begin{equation}\label{pgfrdicy-equation7}
%\left\{
%\begin{aligned}
%&\lim_{k\rightarrow\infty}\exp(\imath (\lambda_2-\lambda_1)t_k)=1,\\
%&\lim_{k\rightarrow\infty}\exp(\imath (\lambda_3-\lambda_1)t_k)=1,\\
%&\lim_{k\rightarrow\infty}\exp(\imath (\lambda_4-\lambda_1)t_k)=1,\\
%&\lim_{k\rightarrow\infty}\exp(\imath (\mu_h-\lambda_1)t_k)=1, &\text{if}~h~\text{is~even},\\
%&\lim_{k\rightarrow\infty}\exp(\imath (\mu_h-\lambda_1)t_k)=\frac{\alpha-\beta}{\alpha+\beta}, &\text{if}~h~\text{is~odd}.\\
%\end{aligned}
%\right.
%\end{equation}
%Let $\Pi_1=\{\lambda_1, \lambda_2,\mu_h~(h ~\text{is~even})\}$ and $\Pi_2=\{\lambda_3,\lambda_4, \mu_h~(h ~\text{is~odd})\}$.
 (\ref{pgfrdicy-equation6}) is equivalent to the system
\begin{equation}\label{pgfrdicy-equation24}
\begin{aligned}
\mid t\lambda_i-\delta_1\mid<\varepsilon \pmod {2\pi},
 ~~&\text{for} ~i=1,2,\\
\mid t\lambda_i-\delta_2\mid<\varepsilon \pmod {2\pi},
~~&\text{for} ~i=3,4,\\
~\mid t\mu_h-\delta_1\mid<\varepsilon \pmod {2\pi},
 ~~&\text{for} ~h \text{~is even},\\
 ~\mid t\mu_h-\delta_2\mid<\varepsilon \pmod {2\pi},
 ~~&\text{for} ~h \text{~is odd}
\end{aligned}
\end{equation}
has a solution $t_\varepsilon$ for all $\varepsilon>0$.
Using the proof of Theorem 2.4 in \cite{ChanDEKL21}, this condition is equivalent to, for any integers $l'_i$ and $l_h$,
$$\sum\limits_{i=2}^4l'_i(\lambda_i-\lambda_1)+\sum\limits_{h=1}^{n-1}l_h(\mu_h-\lambda_1)=0,$$
implies
\begin{equation}\label{pgfrdicy-equation13}
l'_3+l'_4+\sum\limits_{h~\text{odd}}l_h\neq\pm1.
\end{equation}

\noindent\emph{Case 2.} $n$ is even.
%Let $\Pi'_1=\{\lambda_2-\lambda_1,\lambda_3-\lambda_1,\lambda_4-\lambda_1,\mu_h-\lambda_1~(h ~\text{is~even})\}$ and $\Pi'_2=\{\mu_h-\lambda_1~(h ~\text{is~odd})\}$.
 (\ref{pgfrdicy-equation6}) is equivalent to the system
\begin{equation}\label{pgfrdicy-equation15}
\begin{aligned}
\mid t\lambda_i-\delta_1\mid<\varepsilon \pmod {2\pi},
 ~~&\text{for}~1\leq i\leq 4,\\
 ~\mid t\mu_h-\delta_1\mid<\varepsilon \pmod {2\pi},
 ~~&\text{for}~h \text{~is even},\\
~\mid t\mu_h-\delta_2\mid<\varepsilon \pmod {2\pi},
~~&\text{for}~h \text{~is odd}
\end{aligned}
\end{equation}
has a solution $t_\varepsilon$ for all $\varepsilon>0$.
Using the proof of Theorem 2.4 in \cite{ChanDEKL21}, this condition is equivalent to, for any integers $l'_i$ and $l_h$,
$$\sum\limits_{i=2}^4l'_i(\lambda_i-\lambda_1)+\sum\limits_{h=1}^{n-1}l_h(\mu_h-\lambda_1)=0,$$
implies
\begin{equation}\label{pgfrdicy-equation23}
\sum\limits_{h~\text{odd}}l_h\neq\pm1.
\end{equation}

%\subsection{The sufficient conditions for $\Cay(T_{4n}, S)$ admitting PGFR}
%In this section, we study the existence of PGFR on Cayley graphs over dicyclic groups.
%
%\begin{definition}
%Suppose that $n=p^k$, where $p$ is an odd prime number and $k\geq1$.   Assume that
%$S_1=\{a^{\pm k_1},a^{\pm k_2},\ldots,a^{\pm k_r}\},1\leq k_1<\cdots<k_r\leq n-1$. Let $r_{\text{odd}}$ be the number of odd integers in $k_i$, $1\leq i\leq r$. If $(r_{\text{odd}},p)\neq1$, then we say $S_1\in$ \textbf{Case} \textbf{\uppercase\expandafter{\romannumeral1}}.
%\end{definition}

\begin{theorem}\label{pgfrdicy-maintheorem1}
Let $n=p^k$, where $p$ is an odd prime number and $k\geq1$.  Let $S\subseteq T_{4n}$ such that
$S_1=\{a^{\pm k_1},a^{\pm k_2},\ldots,a^{\pm k_r}\},~1\leq k_1<\cdots<k_r\leq n-1,$
$S_2=\langle a \rangle b$ and $S=S_1\cup S_2$. Let $r_{\text{odd}}$ be the number of odd integers in $\{k_i | 1\leq i\leq r\}$. If $(r_{\text{odd}},p)=1$, then $\Cay(T_{4n}, S)$ admits PGFR.
\end{theorem}

\begin{proof}
By a simple calculation, we can easily get $gSg^{-1}=S$ for all $g\in T_{4n}$. Note that the characters of $T_{4n}$ are given in Table 4. By Lemmas \ref{pgfrdicy-lemma0} and \ref{pgfrdicy-lemma1}, the Cayley graph $\Cay(T_{4n}, S)$ has the following eigenvalues:
$$\lambda_1=2r+2p^k, ~\lambda_2=2r-2p^k, ~ \lambda_3=\lambda_4=2r-4r_{\text{odd}}, ~ ~\text{and}~\mu_h=\sum\limits_{i=1}^r\omega^{hk_i}+\omega^{-hk_i}, ~1\leq h\leq p^k-1.$$
Suppose that $l'_i, l_h$ ($2\leq i \leq 4, 1\leq h\leq n-1$) are integers satisfying
$$\sum\limits_{i=2}^4l'_i(\lambda_i-\lambda_1)+\sum\limits_{h=1}^{p^k-1}l_h(\mu_h-\lambda_1)=0,$$
that is,
$$-4p^kl'_2-(l'_3+l'_4)(4r_{\text{odd}}+2p^k)+\sum\limits_{i=1}^r\sum\limits_{h=1}^{p^k-1}l_h\left(\omega^{hk_i}+\omega^{-hk_i}\right)-\sum\limits_{h=1}^{p^k-1}l_h(2r+2p^k)=0.$$
Hence, $\omega$ is a root of the polynomial
\begin{equation}\label{pgfrdicy-equation10}
L_1(x)=-4p^kl'_2-(l'_3+l'_4)(4r_{\text{odd}}+2p^k)+\sum\limits_{i=1}^r\sum\limits_{h=1}^{p^k-1}l_h\left(x^{hk_i}+x^{(2p^k-h)k_i}\right)-\sum\limits_{h=1}^{p^k-1}l_h(2r+2p^k).
\end{equation}
Let $\Phi_{2p^k}(x)$ be the $2p^k$-th cyclotomic polynomial. Then there exists a polynomial $g_1(x)$ such that
\begin{equation}\label{pgfrdicy-equation11}
L_1(x)=\Phi_{2p^k}(x)g_1(x).
\end{equation}
Let $x=-1$. By (\ref{pgfrdicy-equation10}), we have
\begin{equation*}
\begin{aligned}
L_1(-1)&=-4p^kl'_2-(l'_3+l'_4)(4r_{\text{odd}}+2p^k) +2r_{\text{odd}}\sum\limits_{h=1}^{p^k-1}(-1)^hl_h \\ &~~~~+2(r-r_{\text{odd}})\sum\limits_{h=1}^{p^k-1}l_h -\sum\limits_{h=1}^{p^k-1}l_h(2r+2p^k)\\
&=-4p^kl'_2-(l'_3+l'_4)(4r_{\text{odd}}+2p^k)-2p^k\sum\limits_{h~\text{even}}l_h-(4r_{\text{odd}}+2p^k)\sum\limits_{h~\text{odd}}l_h.
\end{aligned}
\end{equation*}
Since $\Phi_{2p^k}(-1)=\Phi_{p^k}(1)=p$. Combining with (\ref{pgfrdicy-equation11}), we have
$$p\mid\left(-4p^kl'_2-(l'_3+l'_4)(4r_{\text{odd}}+2p^k)-2p^k\sum\limits_{h~\text{even}}l_h-(4r_{\text{odd}}+2p^k)\sum\limits_{h~\text{odd}}l_h\right)$$
Then
$$p\mid -4r_{\text{odd}}(l'_3+l'_4+\sum\limits_{h~\text{odd}}l_h),$$
Since $(r_{\text{odd}},p)=1$,  we have
$$l'_3+l'_4+\sum\limits_{h~\text{odd}}l_h\neq \pm1.$$
By (\ref{pgfrdicy-equation13}), the Cayley graph $\Cay(T_{4n}, S)$ admits PGFR between vertices $u$ and $v$~$(0\leq u,v\leq2n-1 ~\text{or}~2n\leq u,v\leq4n-1)$ with $u-v=n$.
\qed\end{proof}

\begin{example}
{\em
Let $n=3$. Suppose that $S=\{a, a^{-1}\}\cup \langle a \rangle b$. Then the eigenvalues of the Cayley graph $\Cay(T_{12}, S)$ are
$$\lambda_1=8,~\lambda_2=-4,~\lambda_3=\lambda_4=-2,~\mu_1=1,~\mu_2=-1.$$
Suppose that $l'_2,l'_3,l'_4, l_1,l_2$ are integers satisfying
$$\sum\limits_{i=2}^4l'_i(\lambda_i-\lambda_1)+\sum\limits_{h=1}^{2}l_h(\mu_h-\lambda_1)=0,$$
that is,
$$-12l'_2-10(l'_3+l'_4)-7l_1-9l_2=0.$$
By a simple calculation, we have
$$-10(l'_3+l'_4+l_1)=12l'_2+9l_2-3l_1.$$
Thus
$$3\mid -10(l'_3+l'_4+l_1),$$
which implies that
$$l'_3+l'_4+l_1\neq \pm1.$$
}
\end{example}
By (\ref{pgfrdicy-equation13}), we get that the Cayley graph $\Cay(T_{12}, S)$ admits PGFR.

We need the Kronecker Approximation Theorem in the following discussion.
\begin{lemma}\emph{(see \cite[Kronecker Approximation Theorem]{Hw})}\label{Kronecker-Approximation-Theorem}
Let $a_1,a_2,\ldots,a_m$ be arbitrary real numbers. Let $1,b_1,b_2,\ldots,b_m$ be linearly independent over $\mathbb{Q}$. Then for $\forall~ \varepsilon>0$, there exist $l\in\mathbb{Z}$ and $q_1,q_2,\ldots, q_m\in \mathbb{Z}$ such that
$$\mid lb_j-q_j-a_j\mid<\varepsilon,~1\leq j\leq s.$$
\end{lemma}

The following lemma is an another version of the Kronecker Approximation Theorem.
\begin{lemma}\emph{(see \cite{Ban,VinetZ2012})}\label{Kronecker-Approximation-Theorem2}
Let $a_1,a_2,\ldots,a_m$  and $b_1,b_2, \ldots,b_m$ be arbitrary real numbers. For $\forall~ \varepsilon>0$, the system
$$|lb_j-a_j|<\varepsilon \pmod {2\pi}, ~j=1,2,\ldots,m,$$
has a solution $l$ if and only if, for $q_1,q_2,\ldots,q_m\in \mathbb{Z}$,
$$q_1b_1+\cdots+q_mb_m=0,$$
implies
$$q_1a_1+\cdots+q_ma_m\equiv0 \pmod {2\pi}.$$
\end{lemma}

Next, we characterize some Cayley graphs over dicyclic groups admitting PGFR when $n$ is a power of two.

\begin{theorem}\label{pgfrdicy-maintheorem2}
Let $n=2^k$, $k\geq2$. Let $S\subseteq T_{4n}$ such that
$S_1=S\cap \langle a \rangle $, $S_2= \langle a \rangle b$ and $S=S_1\cup S_2$.  Then $\Cay(T_{4n}, S)$ admits PGST (that is,  $\Cay(T_{4n}, S)$ admits PGFR) in the following two cases:
\begin{itemize}
\item[\rm(1)] $S_1=\{a^{\pm k_1}\}$ with $k_1$ is odd, $1\leq k_1\leq n-1$;
\item[\rm(2)] $S_1=\{a^{\pm k_1}\}\cup\left(\cup_{j=1}^r\left\{a^{2^{m_j}z}:\gcd(z,2)=1, 1\leq z\leq2^{k-m_j+1}-1\right\}\right)$, where $k_1~(1\leq k_1\leq n-1)$ is odd and $1\leq m_1<\cdots<m_r\leq k$.
\end{itemize}
\end{theorem}

\begin{proof}
By a simple calculation, we can easily verify $gSg^{-1}=S$ for all $g\in T_{4n}$ in two cases as mentioned above.

(1) Assume that $S_1=\{a^{\pm k_1}\}$, where $k_1$ is odd, $1\leq k_1\leq n-1$. Then the eigenvalues of $\Cay(T_{4n}, S)$ are
$$\lambda_1=2+2n,~ \lambda_2=2-2n,~ \lambda_3=\lambda_4=-2,~ \mu_h=\omega^{hk_1}+\omega^{-hk_1}=2\cos\left(\frac{\pi h k_1}{n}\right),~ 1\leq h \leq n-1.$$
%Suppose that $\alpha$ and $\beta$ are two complex numbers with $\alpha+\beta=1$.

%and
%\begin{equation}\label{pgfrdicy-equation15}
%
%\end{equation}

Note that $\mu_h=-\mu_{n-h}$ for $1\leq h \leq 2^{k-1}-1$. It is known that the minimal polynomial of $\omega=\exp(\frac{\pi \imath}{n})$ over $\mathbb{Q}$ is the $2n$-th cyclotomic polynomial $\Phi_{2n}(x)$ with degree $\phi(2n)=2^k$, where $\phi$ is the Euler's phi-function, so is $\omega^{k_1}$ since $\gcd(2,k_1)=1$. We claim that $1, \mu_1,\ldots,\mu_{2^{k-1}-1}$ are linearly independent over $\mathbb{Q}$. Assume that $1, \mu_1, \ldots, \mu_{2^{k-1}-1}$ are linearly dependent over $\mathbb{Q}$. Then there exist $l_0, l_1,\ldots,l_{2^{k-1}-1}\in \mathbb{Q}$ that are not all zero, such that
$$l_0+l_1(\omega^{k_1}+\omega^{-k_1})+\cdots+l_{2^{k-1}-1}(\omega^{k_1(2^{k-1}-1)}+\omega^{-k_1(2^{k-1}-1)})=0.$$
Since
$$\omega^{-k_1 h}=-\omega^{k_1(2^k-h)}, ~1\leq h\leq2^{k-1}-1,$$
then $\omega^{k_1}$ is a root of polynomial
$$L_2(x)=l_0+l_1(x-x^{2^k-1})+l_2(x^2-x^{2^k-2})+\cdots+l_{2^{k-1}-1}(x^{2^{k-1}-1}-x^{2^{k-1}+1}).$$
where $l_0,l_1,\ldots,l_{2^{k-1}-1}$ are not all zero.
Note that the degree of $L_2(x)$ is at most $2^k-1$. But $2^k-1<2^k=\phi(2n)$, a contradiction. Hence $1,\mu_1,\ldots,\mu_{2^{k-1}-1}$ are linearly independent over $\mathbb{Q}$.
%Therefore, $1, \mu_1-\lambda_1,\ldots,\mu_{2^{k-1}-1}-\lambda_1$ are linearly independent over $\mathbb{Q}$.
For $1\leq h\leq2^{k-1}-1$, define
$$a_h=\left\{
\begin{array}{cccc}
0,                 ~~\text{if}~h~\text{is~even},\\[0.2cm]
\frac{1}{2}, ~~\text{if}~h~\text{is~odd}.
\end{array}\right.
$$
By Lemma \ref{Kronecker-Approximation-Theorem}, for $\forall ~\varepsilon>0$, there exist $q, q_1,\ldots,q_{2^{k-1}-1}\in \mathbb{Z}$ such that
\begin{equation}\label{pgfrdicy-equation16}
\mid \mu_hq-q_h-a_h\mid<\frac{\varepsilon}{2\pi}, ~\text{i.e.,}~\mid \mu_h\cdot2\pi q-2\pi q_h-2\pi a_h\mid<\varepsilon.
\end{equation}
Since $\lambda_i$ ($1\leq i\leq 4$) and $\mu_{2^{k-1}}$ are integers, for $\forall ~\varepsilon>0$, there exist integers $q_i' ~(i=1,2,3,4)$ and $q_{2^{k-1}}$ such that
\begin{equation}\label{pgfrdicy-equation14}
\mid\lambda_i\cdot2\pi q-2\pi q_i'\mid<\varepsilon ~~(1\leq i\leq 4) \text{~and~} \mid\mu_{2^{k-1}}\cdot2\pi q-2\pi q_{2^{k-1}}\mid<\varepsilon.
\end{equation}
Combining (\ref{pgfrdicy-equation16}) and (\ref{pgfrdicy-equation14}), we arrive that (\ref{pgfrdicy-equation15}) holds with $\delta_1=0$ and $\delta_2=\pi$. That is, $\Cay(T_{4n},S_1)$ admits PGST with respect to a time sequence in $2\pi \mathbb{Z}$.

(2) Assume that $$S_1=\{a^{\pm k_1}\}\cup\left(\cup_{j=1}^r\left\{a^{2^{m_j}z}:\gcd(z,2)=1, 1\leq z\leq2^{k-m_j+1}-1\right\}\right),$$
where $k_1~(1\leq k_1\leq n-1)$ is odd and $1\leq m_1<\cdots<m_r\leq k$. Then, by Lemma \ref{pgfrdicy-lemma1}, the eigenvalues of $\Cay(T_{4n},S)$ are
$$\lambda_1=2+\sum\limits_{j=1}^r2^{k-m_j}+2^{k+1},~ \lambda_2=2+\sum\limits_{j=1}^r2^{k-m_j}-2^{k+1},~ \lambda_3=\lambda_4=\sum\limits_{j=1}^r2^{k-m_j}-2,$$
and
\begin{align*}
\mu_h&=2\cos(\frac{\pi k_1h}{n})+\frac{1}{2}\sum\limits_{j=1}^r\sum\limits_{\substack{1\leq z\leq 2^{k-m_j+1}-1\\ \gcd(z,2)=1}}(\omega^{2^{m_j}zh}+\omega^{-2^{m_j}zh})\\
     &=2\cos(\frac{\pi k_1h}{n})+\sum\limits_{j=1}^r c(h, 2^{k-m_j+1}).
\end{align*}
where $c(h,n)$ is the Ramanujan function (see \cite[Page 70]{McCarthy86}), which is an integer-valued function (see \cite[Corollary 2.2]{McCarthy86}). Then $1,\mu_1,\ldots,\mu_{2^{k-1}-1}$  are linearly independent over $\mathbb{Q}$. Similar to the proof of (1), we can get the desired result.
\qed\end{proof}
\begin{example}
{\em
Let $n=8$. Suppose that $S=\{a, a^{-1}\}\cup \langle a \rangle b$. Then the eigenvalues of the Cayley graph $\Cay(T_{32}, S)$ are
$$\lambda_1=18,~ \lambda_2=-14,~\lambda_3=\lambda_4=-2, ~ \mu_1=-\mu_7=\sqrt{2+\sqrt{2}},~ \mu_2=-\mu_6=\sqrt{2}, $$
$$\mu_3=-\mu_5=\sqrt{2-\sqrt{2}}.$$
Define $a_1=a_3=0$ and $a_2=\frac{1}{2}$. Since $1,\sqrt{2+\sqrt{2}},\sqrt{2}$ and $\sqrt{2-\sqrt{2}}$ are linearly independent over $\mathbb{Q}$, by Kronecker Approximation Theorem, for $\varepsilon>0$, there exist $q, q_1,q_2,q_3$ such that
\begin{equation*}
\mid \mu_h\cdot2\pi q-2\pi q_h-2\pi a_h\mid<\varepsilon, ~~h=1,2,3.
\end{equation*}
Since $\lambda_i$ ($1\leq i\leq 4$) and $\mu_4$ are integers, for $\forall ~\varepsilon>0$, there exist integers $q_i$ such that
\begin{equation*}
\mid\lambda_i\cdot2\pi q-2\pi q_i\mid<\varepsilon ~~(1\leq i\leq 4) \text{~and~} \mid\mu_4\cdot2\pi q-2\pi q_4\mid<\varepsilon.
\end{equation*}
By (\ref{pgfrdicy-equation15}), $\Cay(T_{32},S)$ admits PGST (that is,  $\Cay(T_{32}, S)$ admits PGFR) with respect to a time sequence in $2\pi \mathbb{Z}$.
}
\end{example}
For any integer $n$ and a prime number $p$, if $p^s|n$ and $p^{s+1}\nmid n$, $s\geq 0$, then we write $v_p(n)=s$. Let $\mathbb{N}$ and $\mathbb{Z}$ be the set of non-negative integers and the set of integers, respectively.
%\begin{definition}
%Suppose that $n=pqm$, where $p,q$ are distinct odd prime numbers and $m\in \mathbb{N}$.  Assume that
%$S_1=\{a^{\pm k_1},a^{\pm k_2},\ldots,a^{\pm k_r}\}\varsubsetneqq  \langle a \rangle $. If there exists integers $i, j$ such that $v_p(k_i)\neq v_p(k_j)$ or $v_q(k_i)\neq v_q(k_j)$ or $v_p(k_i)=v_p(n)$ or $v_q(k_i)=v_q(n)$, then we say $S_1\in$ \textbf{Case} \textbf{\uppercase\expandafter{\romannumeral2}}.
%\end{definition}

%\subsection{Cayley graphs doesn't admit PGFR}
\begin{theorem}\label{pgfrdicy-maintheorem3}
Let $n=pqm$, where $p,q$ are distinct odd prime numbers and $m\in \mathbb{N}$. Let $S\subseteq T_{4n}$ such that
$S_1=S\cap \langle a \rangle =\{a^{\pm k_1},a^{\pm k_2},\ldots,a^{\pm k_r}\}, 1\leq k_1<\cdots<k_r\leq n-1$, $S_2= \langle a \rangle b$ and $S=S_1\cup S_2$.
If $v_p(k_1)=v_p(k_2)=\cdots=v_p(k_r)<v_p(n)$ and $v_q(k_1)=v_q(k_2)=\cdots=v_q(k_r)<v_q(n)$,
then $\Cay(T_{4n},S)$ does not admit PGFR.

\end{theorem}

\begin{proof}
We prove this theorem by contradiction. Assume that $\Cay(T_{4n},S)$ admits PGFR.
By Lemma \ref{pgfrdicy-lemma1},
$$\mu_h=2\sum\limits_{i=1}^r\cos\left(\frac{\pi k_i h}{n}\right), ~1\leq h\leq n-1$$
are some eigenvalues of $\Cay(T_{4n},S)$.
%We further discuss it according to either $p|k_1$ or $p\nmid k_1$, respectively.
Let $\omega_p=\exp(\frac{2\pi \imath}{p})$ be a $p$-th root of unity. Suppose that $v_p(k_i)=s$, $1\leq i\leq r$. Then $\gcd\left(\frac{k_i}{p^s}, p\right)=1$, $1\leq i\leq r$. So $\omega_p^{\frac{k_i}{p^s}}$ is also a  $p$-th root of unity. Note that $\sum\limits_{j=0}^{p-1}\omega_p^{\frac{k_i}{p^s}j}=0$, that is,
\begin{equation}\label{pgfrdicy-equation17}
1+2\sum\limits_{j=1}^{\frac{p-1}{2}}\cos\left(\frac{2\pi j\frac{k_i}{p^s} }{p}\right)=0, ~1\leq i\leq r.
\end{equation}
Multiplying both sides of (\ref{pgfrdicy-equation17}) by $2\cos\left(\frac{\pi k_i}{n}\right)$, we have
$$2\cos\left(\frac{\pi k_i}{n}\right)+2\sum\limits_{j=1}^{\frac{p-1}{2}}\left(\cos\left(\frac{\pi k_i(2\frac{mq}{p^s}j+1)}{n}\right)+\cos\left(\frac{\pi k_i(2\frac{mq}{p^s}j-1)}{n}\right)\right)=0, ~1\leq i\leq r.$$
Adding up the above equations from $i=1$ to $i=r$, we get
\begin{equation*}
\mu_1+\sum\limits_{j=1}^{\frac{p-1}{2}}\left(\mu_{2\frac{mq}{p^s}j+1}+\mu_{2\frac{mq}{p^s}j-1}\right)=0.
\end{equation*}
Since $v_p(k_i)<v_p(n)$,  $\frac{mq}{p^s}$ is an integer. Recall that $\alpha+\beta=\exp(\imath \delta_1)$ and $\alpha-\beta=\exp(\imath \delta_2)$. By (\ref{pgfrdicy-equation6}) and Lemma \ref{Kronecker-Approximation-Theorem2}, we have
\begin{equation}\label{pgfrdicy-equation18}
p\delta_2\equiv0 \pmod {2\pi}.
\end{equation}
Multiplying both sides of (\ref{pgfrdicy-equation17}) by $2\cos\left(\frac{2\pi k_i}{n}\right)$, we have
$$2\cos\left(\frac{2\pi k_i}{n}\right)+2\sum\limits_{j=1}^{\frac{p-1}{2}}\left(\cos\left(\frac{\pi k_i(2\frac{mq}{p^s}j+2)}{n}\right)+\cos\left(\frac{\pi k_i(2\frac{mq}{p^s}j-2)}{n}\right)\right)=0, ~1\leq i\leq r.$$
Adding up the above equations from $i=1$ to $i=r$, we get
\begin{equation*}
\mu_2+\sum\limits_{j=1}^{\frac{p-1}{2}}\left(\mu_{2\frac{mq}{p^s}j+2}+\mu_{2\frac{mq}{p^s}j-2}\right)=0.
\end{equation*}
By (\ref{pgfrdicy-equation6}) and Lemma \ref{Kronecker-Approximation-Theorem2}, we have
\begin{equation}\label{pgfrdicy-equation19}
p\delta_1\equiv0 \pmod {2\pi}.
\end{equation}
Combining (\ref{pgfrdicy-equation18}) and (\ref{pgfrdicy-equation19}), we have
\begin{equation*}\label{pgfrdicy-equation20}
p(\delta_1-\delta_2)\equiv0 \pmod {2\pi}.
\end{equation*}
%If $\Cay(T_{4n},S)$ admits PGFR, by (\ref{pgfrdicy-equation6}) and (\ref{pgfrdicy-equation7}), then
%\begin{equation}
%\begin{aligned}
%\mid t_\varepsilon(\mu_h-\mu_{h-1})-\mu\mid<\varepsilon \pmod {2\pi} ~~&\text{if}~~h ~\text{is~even},\\
%\mid t_\varepsilon(\mu_h-\mu_{h-1})+\mu\mid<\varepsilon \pmod {2\pi} ~~&\text{if}~~h ~\text{is~odd}.\\
%\end{aligned}
%\end{equation}
%Appliying it to (\ref{pgfrdicy-equation20}) yields
%$$p\mu\equiv0 \pmod {2\pi}. $$
Similarly, we have
$$q(\delta_1-\delta_2)\equiv0 \pmod {2\pi}. $$
Since $p$ and $q$ are distinct prime numbers, there exist integers $s,t$ such that $sp+tq=1$. Then
$$\delta_1-\delta_2=sp(\delta_1-\delta_2)+tq(\delta_1-\delta_2)\equiv0 \pmod {2\pi}, $$
which implies $\beta=0$, a contradiction.
%\noindent\emph{Subcase 1.2.} $p\mid k_1$. Suppose that $v_p(k_1)=s$.
%Since $v_p(k_i)<v_p(n)$, then $\frac{mq}{p^s}$ is an integer. With the similarly procedure in Case 1.1, we get
%\begin{equation}\label{pgfrdicy-equation21}
%\mu_1+\sum\limits_{j=1}^{\frac{p-1}{2}}\left(\mu_{2\frac{mq}{p^s}j+1}+\mu_{2\frac{mq}{p^s}j-1}\right)=0
%\text{~and~}
%\mu_2+\sum\limits_{j=1}^{\frac{p-1}{2}}\left(\mu_{2\frac{mq}{p^s}j+2}+\mu_{2\frac{mq}{p^s}j-2}\right)=0.
%\end{equation}
%
%
%With the same discussion as in Case 1.1, we also lead to a contradiction.
%
%\noindent\emph{Case 2.} $r>1$. In this case, some eigenvalues of $\Cay(T_{4n},S)$ are
%$$\mu_h=2\sum\limits_{i=1}^r\cos\left(\frac{\pi k_i h}{n}\right), 1\leq h\leq n-1.$$
%Suppose that $v_p(k_i)=s, i=1,2,\ldots,r$. With the same discussion in Case 1, we have
%\begin{align*}
%\cos\left(\frac{\pi k_i}{n}\right)+\sum\limits_{j=1}^{\frac{p-1}{2}}\left(\cos\left(\frac{\pi k_i(2\frac{mq}{p^s}j+1)}{n}\right)+\cos\left(\frac{\pi k_i(2\frac{mq}{p^s}j-1)}{n}\right)\right)=0, 1\leq i\leq r.
%\end{align*}
%and
%\begin{align*}
%\cos\left(\frac{2\pi k_i}{n}\right)+\sum\limits_{j=1}^{\frac{p-1}{2}}\left(\cos\left(\frac{\pi k_i(2\frac{mq}{p^s}j+2)}{n}\right)+\cos\left(\frac{\pi k_i(2\frac{mq}{p^s}j-2)}{n}\right)\right)=0, 1\leq i\leq r.
%\end{align*}
%Sum the equation for every $1\leq i\leq r$, we can get (\ref{pgfrdicy-equation21}), and then also lead to a contradiction.
\qed\end{proof}

The proof of the following result is similar to that of \cite[Lemma 4.3.2]{vanBommelphd19}. Hence we omit the detail here.

\begin{lemma}
Let $p>1$ and $q\ge1$ be two odd integers, and $n=kp$ with $k$ a positive integer.  Suppose that $0\leq a<k$ is an integer. Then
\begin{equation}\label{pgfrdicy-equation22}
\sum_{j=0}^{p-1}(-1)^j\cos\left(\frac{(a+jk)q\pi}{n}\right)=0.
\end{equation}
\end{lemma}

%\begin{definition}
%Let $n=2^sp^t$, where $p$ is an odd prime integer and $s>1$.  Assume that
%$S_1=\{a^{\pm k_1},a^{\pm k_2},\ldots,a^{\pm k_r}\}\subsetneqq  \langle a \rangle $. If there exists integers $i, j$ such that $v_2(k_i)\neq v_2(k_j)$ or $v_2(k_i)=s$ or $v_p(k_i)=t$, then we say $S_1\in$ \textbf{Case} \textbf{\uppercase\expandafter{\romannumeral3}}.
%\end{definition}

\begin{theorem}\label{pgfrdicy-maintheorem4}
Let $n=2^sp$, where $p$ is an odd integer and $s\geq1$ is an integer. Let $S\subseteq T_{4n}$ such that
$S_1=\{a^{\pm k_1},a^{\pm k_2},\ldots,a^{\pm k_r}\}, 1\leq k_1<\cdots<k_r\leq n-1$, $S_2=\langle a \rangle b$ and $S=S_1\cup S_2$.
If $v_2(k_1)=v_2(k_2)=\cdots=v_2(k_r)=s'<s$, then $\Cay(T_{4n},S)$ does not admit PGFR.
\end{theorem}

\begin{proof}
Suppose that $k_i=2^{s'}q_i$, where $q_i$ is an odd integer, for $1\leq i\leq r$. Let $k=2^{s-s'}$. By (\ref{pgfrdicy-equation22}), we have
\begin{equation*}
\sum\limits_{j=0}^{p-1}(-1)^j\cos\left(\frac{(a+j2^{s-s'})q_i\pi}{2^{s-s'}p}\right)=0, ~1\leq i\leq r.
\end{equation*}
Then
\begin{equation*}
\sum\limits_{i=1}^{r}\sum\limits_{j=0}^{p-1}(-1)^j\cos\left(\frac{(a+j2^{s-s'})q_i\pi}{2^{s-s'}p}\right)=0.
\end{equation*}
Note Lemma \ref{pgfrdicy-lemma1} that
$$\mu_h=2\sum_{i=1}^r\cos\left(\frac{\pi k_i h}{n}\right)=2\sum_{i=1}^r\cos\left(\frac{\pi q_i h}{2^{s-s'}p}\right)$$
are some eigenvalues of $\Cay(T_{4n},S)$.
Therefore
$$\sum\limits_{j=0}^{p-1}(-1)^j\mu_{a+j2^{s-s'}}=0.$$
Since
$$\sum\limits_{j=0}^{p-1}(-1)^j(\mu_{a+j2^{s-s'}}-\lambda_1)=\sum\limits_{j=0}^{p-1}(-1)^j\mu_{a+j2^{s-s'}}-\sum\limits_{j=0}^{p-1}(-1)^j\lambda_1=-\lambda_1,$$
let $a=1$ or $a=2$, then
$$\sum\limits_{j=0}^{p-1}(-1)^j(\mu_{1+j2^{s-s'}}-\lambda_1)=-\lambda_1,
\text{~~and~~}
\sum\limits_{j=0}^{p-1}(-1)^j(\mu_{2+j2^{s-s'}}-\lambda_1)=-\lambda_1.$$
Define
\begin{equation*}
l_h=\left\{
\begin{array}{rcl}
(-1)^j, &~&\text{if}~h=2+j2^{s-s'}, ~j=0,1, \ldots, p-1,\\[0.2cm]
(-1)^{j+1}, &~&\text{if}~h=1+j2^{s-s'},~ j=0,1, \ldots, p-1,\\[0.2cm]
0, &~&\text{otherwise},
\end{array}\right.
\end{equation*}
 and $l'_i=0,~ 2\leq i\leq4$. Notice that
$$\sum\limits_{i=2}^4l'_i(\lambda_i-\lambda_1)+\sum\limits_{h=1}^{n-1}l_h(\mu_h-\lambda_1)=0,$$
%implies
%\begin{equation}\label{pgfrdicy-equation23}
%\sum\limits_{h~\text{odd}}l_h\neq\pm1.
%\end{equation}
but
$$\sum\limits_{h~\text{odd}}l_h=-1,$$
a contradiction to (\ref{pgfrdicy-equation23}). Thus, $\Cay(T_{4n},S)$ does not admit PGFR.
\qed
\end{proof}

See the following theorem that if $S$ does not satisfy the conditions of Theorems \ref{pgfrdicy-maintheorem3} and \ref{pgfrdicy-maintheorem4}, then $\Cay(T_{4n},S)$ may admit PGFR.
\begin{theorem}
Let $n=p^tm~(t\geq1)$, where $p$ is  an odd prime number and $m$ is a positive integer. Let $S=\{a^{m}, a^{-m}\}\cup \langle a \rangle b$. Then $\Cay(T_{4n}, S)$ admits PGFR.
\end{theorem}
\begin{proof}
Write $\omega_{2p^t}=\exp(\frac{\pi \imath}{p^t})$. We consider the following two cases.

\noindent\emph{Case 1.} $m$ is odd.
Then the eigenvalues of $\Cay(T_{4n}, S)$ are
$$\lambda_1=2+2p^tm,~\lambda_2=2-2p^tm,~\lambda_3=\lambda_4=-2, ~ \mu_h=\omega_{2p^t}^h+\omega_{2p^t}^{-h}, ~1\leq h\leq p^tm-1.$$
Suppose that $l'_i, l_h$ ($2\leq i \leq 4, 1\leq h\leq p^tm-1$) are integers satisfying
$$\sum\limits_{i=2}^4l'_i(\lambda_i-\lambda_1)+\sum\limits_{h=1}^{p^tm-1}l_h(\mu_h-\lambda_1)=0,$$
that is,
$$-4p^tml'_2-(l'_3+l'_4)(4+2p^tm)+ \sum\limits_{h=1}^{p^tm-1}l_h\left(\omega_{2p^t}^{h} +\omega_{2p^t}^{-h}\right)-\sum\limits_{h=1}^{p^tm-1}l_h(2+2p^tm)=0.$$
Hence, $\omega_{2p^t}$ is a root of the polynomial
\begin{equation}\label{pgfrdicy-equation25}
L_3(x)=-4p^tml'_2-(l'_3+l'_4)(4+2p^tm) +\sum\limits_{h=1}^{p^tm-1}l_h\left(x^{h}+x^{-h}\right) -\sum\limits_{h=1}^{p^tm-1}l_h(2+2p^tm).
\end{equation}
Let $\Phi_{2p^t}(x)$ be the $2p^t$-th cyclotomic polynomial. Then there exists a polynomial $g_3(x)$ such that
\begin{equation}\label{pgfrdicy-equation26}
L_3(x)=\Phi_{2p^t}(x)g_3(x).
\end{equation}
Let $x=-1$. By (\ref{pgfrdicy-equation25}), we have
\begin{equation*}
\begin{aligned}
L_3(-1)=-4p^tml'_2-(l'_3+l'_4)(4+2p^tm) -2p^tm\sum\limits_{h~\text{even}}l_h -(4+2p^tm)\sum\limits_{h~\text{odd}}l_h.
\end{aligned}
\end{equation*}
Since $\Phi_{2p^t}(-1)=\Phi_{p^t}(1)=p$. Combining with (\ref{pgfrdicy-equation26}), we have $p\mid L_3(-1)$. Then
$$p\mid -4(l'_3+l'_4+\sum\limits_{h~\text{odd}}l_h),$$
which implies that
$$l'_3+l'_4+\sum\limits_{h~\text{odd}}l_h\neq \pm1.$$
By (\ref{pgfrdicy-equation13}),  $\Cay(T_{4n}, S)$ admits PGFR.

\noindent\emph{Case 2.} $m$ is even.
Then the eigenvalues of $\Cay(T_{4n}, S)$ are
$$\lambda_1=2+2p^tm,~\lambda_2=2-2p^tm,~\lambda_3=\lambda_4=2, ~ \mu_h=\omega_{2p^t}^h+\omega_{2p^t}^{-h}, ~1\leq h\leq p^tm-1.$$
Similar to Case 1, suppose that $l'_i, l_h$ ($2\leq i \leq 4, 1\leq h\leq p^tm-1$) are integers satisfying
$$\sum\limits_{i=2}^4l'_i(\lambda_i-\lambda_1)+\sum\limits_{h=1}^{p^tm-1}l_h(\mu_h-\lambda_1)=0,$$
that is,
$$-4p^tml'_2-2p^tm(l'_3+l'_4)+ \sum\limits_{h=1}^{p^tm-1}l_h\left(\omega_{2p^t}^{h} +\omega_{2p^t}^{-h}\right)-\sum\limits_{h=1}^{p^tm-1}l_h(2+2p^tm)=0.$$
Hence, $\omega_{2p^t}$ is a root of the polynomial
\begin{equation}\label{pgfrdicy-equation36}
L_4(x)=-4p^tml'_2-2p^tm(l'_3+l'_4) +\sum\limits_{h=1}^{p^tm-1}l_h\left(x^{h}+x^{-h}\right) -\sum\limits_{h=1}^{p^tm-1}l_h(2+2p^tm).
\end{equation}
Then there exists a polynomial $g_4(x)$ such that
\begin{equation}\label{pgfrdicy-equation37}
L_4(x)=\Phi_{2p^t}(x)g_4(x).
\end{equation}
Let $x=-1$. By (\ref{pgfrdicy-equation36}), we have
\begin{equation*}
\begin{aligned}
L_4(-1)=-4p^tml'_2-2p^tm(l'_3+l'_4) -2p^tm\sum\limits_{h~\text{even}}l_h -(4+2p^tm)\sum\limits_{h~\text{odd}}l_h.
\end{aligned}
\end{equation*}
Recall that $\Phi_{2p^t}(-1)=\Phi_{p^t}(1)=p$. Combining with (\ref{pgfrdicy-equation37}), we have $p\mid L_4(-1)$. Then
$$p\mid -4\sum\limits_{h~\text{odd}}l_h,$$
which implies that
$$\sum\limits_{h~\text{odd}}l_h\neq \pm1.$$
By (\ref{pgfrdicy-equation23}),  $\Cay(T_{4n}, S)$ admits PGFR.
\qed
\end{proof}

\section{Conclusions}
In this paper, we first give a necessary and sufficient description for $\Cay(T_{4n}, S)$ admitting PGFR by analysing the spectral decomposition of the transition matrix of $\Cay(T_{4n}, S)$. By this description, we give some sufficient conditions for $\Cay(T_{4n}, S)$ admitting PGFR when $n$ is a power of a prime number, or $n=p^tm~(t\geq1)$ with $p$ an odd prime number and $m$ a positive integer. Also we give some sufficient conditions for $\Cay(T_{4n}, S)$ not admitting PGFR when $n=pqm$ with $p,q$ distinct odd prime numbers and $m\in \mathbb{N}$, or $n=2^sp$ with $p$ an odd integer and $s\geq1$ an integer. We would like to mention that the method used in this paper can be applied to characterize Cayley graphs over other finite non-abelian groups admitting PGFR.
%\begin{itemize}
%\item[\rm(1)]Given a necessary and sufficient condition for Cayley graphs over dicyclic groups admitting PGFR;
%\item[\rm(2)] Research all  Cayley graphs over finite groups admitting PGFR.
%\end{itemize}

\end{document}